\documentclass[11pt]{amsart}
\usepackage{cases}

\theoremstyle{plain}
\newtheorem{theorem}                {Theorem}      [section]
\newtheorem*{theorem*}                {Theorem \ref{thm:appl}}
\newtheorem{proposition}  [theorem]  {Proposition}

\newtheorem{lemma}        [theorem]  {Lemma}

\theoremstyle{definition}

\newtheorem{remark}       [theorem]  {Remark}

\def \n{\mbox{${\mathbb N}$}}

\renewcommand{\i}{\mathrm{i}}

\DeclareMathOperator{\trace}{trace} 

\DeclareMathOperator{\rank}{rank}

\numberwithin{equation}{section}

\begin{document}

\title[Biharmonic tori in spheres]
{Biharmonic tori in spheres}

\author{Dorel~Fetcu}
\author{Eric~Loubeau}
\author{Cezar~Oniciuc}

\address{Department of Mathematics and Informatics\\
Gh. Asachi Technical University of Iasi\\
Bd. Carol I, 11 \\
700506 Iasi, Romania} \email{dfetcu@math.tuiasi.ro}

\address{D{\'e}partement de Math{\'e}matiques \\
LMBA, UMR 6205 \\
Universit{\'e} de Bretagne Occidentale \\
6, avenue Victor Le Gorgeu \\
CS 93837, 29238 Brest Cedex 3, France}
\email{Eric.Loubeau@univ-brest.fr}

\address{Faculty of Mathematics\\ Al. I. Cuza University of Iasi\\
Bd. Carol I, 11 \\ 700506 Iasi, Romania} \email{oniciucc@uaic.ro}

\thanks{The first and third authors' work was supported by a grant of the Romanian National Authority for Scientific Research and Innovation, CNCS - UEFISCDI, project number PN-II-RU-TE-2014-4-0004.}



\begin{abstract} We consider proper-biharmonic flat tori with constant mean curvature (CMC) in spheres and find necessary and sufficient conditions for certain rectangular tori and square tori to admit full CMC proper-biharmonic immersions in $\mathbb{S}^n$, as well as the explicit expressions of some of these immersions.
\end{abstract}

\maketitle

\section{Introduction}

Suggested in $1964$ in the seminal paper \cite{ES} the notion of a biharmonic map between two Riemannian manifolds, defined as a critical point of the bienergy functional
$$
E_2 (\phi) = \frac{1}{2} \int_M |\tau(\phi)|^2 \, v_g,
$$
for $\phi : (M,g) \to (N,\tilde{g})$, represents a generalization of harmonic maps. Since any harmonic map is biharmonic, we study proper-biharmonic maps, i.e., biharmonic maps which are not harmonic. The Euler-Lagrange equation corresponding to the bienergy, obtained in \cite{J}, is given by the vanishing of the bitension field, i.e.,
$$
\tau_2 (\phi) = - \Delta \tau(\phi) - \trace R^N (d\phi,\tau(\phi))d\phi =0.
$$
We note that this definition of biharmonicity coincides with the one in \cite{C} only for the case of immersions in Euclidean space. 

Specializing to Riemannian immersions, or submanifolds, the relationship between minimality and harmonicity is a well-known fact. A special connection, although a more complex one, does also exist between biharmonic submanifolds and submanifolds with constant mean curvature, called CMC submanifolds. This relation was investigated in many papers, one of the most recent being \cite{LO}. In this paper are described CMC proper-biharmonic immersions of surfaces in spheres, a result which led to the fact that for any $h\in(0,1)$ there exist CMC proper-biharmonic planes and cylinders in $\mathbb{S}^5$ with mean curvature $h$ and also to the finding of a necessary and sufficient condition, in terms of $h$, for the existence of CMC proper-biharmonic tori in $\mathbb{S}^5$. We note that a result in \cite{O} shows that the value of the mean curvature of a CMC proper-biharmonic submanifold in a unit Euclidean sphere $\mathbb{S}^n$ must be less or equal to one. Moreover, the mean curvature is equal to one if and only if the submanifold lies in a small hypersphere $\mathbb{S}^{n-1}(1/\sqrt{2})$ as a minimal submanifold. The case of proper-biharmonic surfaces with parallel (in the normal bundle) mean curvature vector field proved to be a rather rigid one as it turned out that the mean curvature of such surfaces must be equal to one (see \cite{BMO}). 

In the first part of our paper we revisit the classification of CMC proper-bi\-har\-mo\-nic flat tori in $\mathbb{S}^5$ and give a new proof of the result in \cite{LO}. Then, looking from a different angle, we ask whether a given torus admits a full CMC proper-biharmonic immersion in a sphere $\mathbb{S}^n$. We study this problem for two families of tori: rectangular tori with a side of length equal to one and square tori. In the first case we determine the expressions of all such rectangular tori that admit full CMC proper-biharmonic immersions in a sphere $\mathbb{S}^n$, as well as these immersions. This result is a rather rigid one as the only admissible dimensions for the ambient space are $5$ and $7$. The case of square tori is more flexible and we prove that there are examples of full CMC proper-biharmonic immersions of such tori in $\mathbb{S}^n$ for any odd positive integer $n\geq 5$.    

We will always consider full immersions in $\mathbb{S}^n$, i.e., their images do not lie in any totally geodesic sphere $\mathbb{S}^{n'}\subset \mathbb{S}^n$.

{\bf Acknowledgments.} The authors would like to thank Iulian Stoleriu for verifying some of the computations in Proposition \ref{prop12}. 

\section{General results}

We will first recall a general characterization theorem on CMC proper-bi\-har\-mo\-nic flat surfaces in spheres obtained in \cite{LO} as a direct application of a result in \cite{M} and then specialize our study to the case of tori and give a new proof to an existence criterion in \cite{LO} for CMC proper-biharmonic tori in $\mathbb{S}^5$, formulated in terms of their mean curvature.

\begin{theorem}[\cite{LO}]\label{thmM}
Let $D$ be a small disk about the origin in the Euclidean plane $\mathbb{R}^2$ and $\phi : D \to \mathbb{S}^{n}$ be a CMC proper-biharmonic immersion with mean curvature $h=|H|\in (0,1)$. Then
\begin{enumerate}
\item[(i)] $n$ is odd, $n\geq 5$.
\item[(ii)] $\phi$ extends uniquely to a CMC proper-biharmonic immersion of $\mathbb{R}^2$ into $\mathbb{S}^{n}$.
\item[(iii)] $\psi = i \circ \phi : \mathbb{R}^2 \to \mathbb{R}^{n+1}$ can be written
\begin{align}\label{psi}
\psi (z) =& \tfrac{1}{\sqrt{2}} \sum_{k=1}^{m} \sqrt{R_k}\left(  e^{\tfrac{\sqrt{\lambda_1}}{2}(\mu_k z - \bar{\mu}_k\bar{z})}Z_k
+ e^{\tfrac{\sqrt{\lambda_1}}{2}(-\mu_k z + \bar{\mu}_k\bar{z})}\bar{Z}_k\right)\\\nonumber
&+ \tfrac{1}{\sqrt{2}} \sum_{j=1}^{m'} \sqrt{R'_j}\left( e^{\tfrac{\sqrt{\lambda_2}}{2}(\eta_j z - \bar{\eta}_j \bar{z})}W_j
+ e^{\tfrac{\sqrt{\lambda_2}}{2}(-\eta_j z + \bar{\eta}_j \bar{z})}\overline{W}_j\right) ,
\end{align}
where 
\begin{enumerate}
\item $Z_k = \left( E_{2k-1} - \i E_{2k}\right)/2, \, k\in\{1,\dots,m\}, \i^2=-1$,
\item $W_j = \left( E_{2(m+j)-1} - \i E_{2(m+j)}\right)/2 ,\, j\in\{1,\dots,m'\},$
\item $\{E_1,\dots,E_{2m+2m'}\}$ is an orthonormal basis of $\mathbb{R}^{n+1}$, $n=2m+2m'-1$,
\item $\lambda_1 = 2(1-h)$, $\lambda_2 = 2(1+h)$,  
\item $\sum_k R_k =1$, $\sum_j R'_j =1$, $R_k >0$, $R'_j >0$,
\item $(1-h)\sum_k \mu^2_k R_k + (1+h)\sum_j \eta^2_jR'_j = 0$.
\item $\{ \pm \mu_k\}_{k=1}^{m}$ are $2m$ distinct complex numbers of norm $1$, 
\item $\{ \pm \eta_j\}_{j=1}^{m'}$ are $2m'$ distinct complex numbers of norm $1$.
\end{enumerate}
\end{enumerate}
\end{theorem}

\begin{remark}[symmetries of solutions]
Let 
$$
(h,(R_k)_k,(R'_j)_j,(\mu_k)_k,(\eta_j)_j)=(R_k,R'_j,\mu_k,\eta_j)
$$ 
be a solution of (e), (f), (g) and (h). Then
\begin{itemize}
\item $(R_k,R'_j,\pm\mu_k,\pm\eta_j)$ is also a solution of Conditions~(e), (f), (g), and (h) ($2^{m+m'}$ solutions).
\item $(R_k,R'_j,\bar{\mu}_k,\bar{\eta}_j)$ is also a solution of Conditions~(e), (f), (g), and (h) (1 solution). 
\item $(R_k,R'_j,\alpha \mu_k,\alpha\eta_j),\alpha\in\mathbb{C}, |\alpha|=1$, is also a solution of Conditions~(e), (f), (g), and (h).
\item $\left(R_{\sigma(k)},R'_{\sigma'(j)},\mu_{\sigma(k)},\eta_{\sigma'(j)}\right)$, where $\sigma$ and $\sigma'$ are permutations of $\{1,\ldots,m\}$ and $\{1,\ldots,m'\}$, respectively, is also a solution of Conditions~(e), (f), (g), and (h) ($m!m'!$ solutions).
\end{itemize}
Each of the above solutions corresponds to an orthogonal linear transformation of $\mathbb{R}^2$ or $\mathbb{R}^n$. Since the classification of isometric immersions is up to isometries of the domain and codomain, our classification will be modulo the above four transformations.
\begin{itemize}
\item if we consider a translation $(x,y)\to(x+a,y+b)$ of the domain, then one obtains a new CMC proper-biharmonic immersion, corresponding to the same data set $(R_k,R'_j,\mu_k,\eta_j)$ but with different $Z_k$'s and $W_j$'s.
\end{itemize}
\end{remark}

\begin{remark} The numbers $m$ and $m'$ are invariant under isometries of $\mathbb{R}^2$ and $\mathbb{S}^n$.
\end{remark}

\begin{theorem}[\cite{LO}, Structure theorem in $\mathbb{S}^5$] \label{thm6}
For a given $h\in (0,1)$ there is a one-parameter family of CMC proper-biharmonic surfaces $\phi_{h,\rho}=\phi_{\rho}:\mathbb{R}^2\to\mathbb{S}^5$ with mean curvature $h$,
$\rho\in[0,(1/2)\arccos((h-1)/(1+h))]$, such that $\psi_{\rho}=i\circ\phi_{\rho}:\mathbb{R}^2 \to \mathbb{R}^6$ can be written as
\begin{align*}
\psi_{\rho}(z) = & \frac{1}{\sqrt{2}}\left( e^{\tfrac{\sqrt{\lambda_1}}{2}(z-\bar{z})}Z_1
+e^{\tfrac{\sqrt{\lambda_1}}{2}(-z+\bar{z})}\bar{Z}_1\right) \\
& +\frac{1}{\sqrt{2}} \sum_{j=1}^2\sqrt{R'_j}\left(e^{\tfrac{\sqrt{\lambda_2}}{2}(\eta_j z-\bar{\eta}_j\bar{z})}W_j
+ e^{\tfrac{\sqrt{\lambda_2}}{2}(-\eta_j z+\bar{\eta}_j
\bar{z})}\bar{W}_j\right) ,
\end{align*}
where
\begin{enumerate}
\item[(a)] $Z_1=\frac{1}{2}\left(E_1-\i E_2\right)$, 
\item[(b)] $W_j=\frac{1}{2}\left(E_{2(1+j)-1}-\i E_{2(1+j)}\right)$, \ $j\in\{1,2\}$, 
\item[(c)] $\{E_1,\ldots,E_6\}$ is an orthonormal basis of $\mathbb{R}^6$, 
\item[(d)] $\lambda_1=2(1-h)$, $\lambda_2=2(1+h)$,
\end{enumerate}
and $R'_1$, $R'_2$, $\eta_1=e^{\i\rho}$ and $\eta_2=e^{\i\tilde{\rho}}$ are given by
$$
\left(\frac{1-\left(\frac{1-h}{h+1}\right)^2}{2\left(1+ \frac{1-h}{h+1}\cos 2\rho\right)}, \ 1-\frac{1-\left(\frac{1-h}{h+1}\right)^2}{2\left(1+ \frac{1-h}{h+1}\cos 2\rho\right)}, \ \rho, \ \tilde{\rho}=\arctan{(-\frac{1}{h\tan{\rho}})} \right),
$$
if $\rho\in(0,(1/2)\arccos((h-1)/(1+h))]$, and
$$
\left(\frac{h}{1+h}, \ \frac{1}{1+h}, \ 0, \ -\frac{\pi}{2}\right),
$$
if $\rho=0$.

Conversely, let $\phi:\mathbb{R}^2\to\mathbb{S}^5$ be a CMC proper-biharmonic surface with mean curvature $h\in (0,1)$. Then, up to isometries of $\mathbb{R}^2$ and $\mathbb{S}^5$, $\psi=i\circ\phi:\mathbb{R}^2 \to \mathbb{S}^5$ is one of the above maps.
\end{theorem}

The proof of the above theorem is based on the fact that we cannot take $m=2$ and $m'=1$, as there are no solutions of Conditions~(e), (f), (g), and (h) in Theorem~\ref{thmM}. The solution in the remaining case $m=1$ and $m'=2$, follows from the next lemma.

\begin{lemma}[\cite{LO}] \label{lemma1}
Let $\eta_1 = e^{\i \rho}$ with, because of symmetries of solutions, $\rho \in [0,\pi/2]$ and $\eta_2 = e^{\i \tilde{\rho}}$, $\tilde{\rho} \in [-\pi/2 , \pi/2)$, then 
\begin{equation*}
\tilde{\rho} = 
\begin{cases}
-\tfrac{\pi}{2},\quad\textnormal{if} \quad\rho=0 \\
0,\quad\textnormal{if} \quad\rho= \tfrac{\pi}{2} \\
\arctan\left( \frac{-1}{h\tan\rho}\right),\quad\textnormal{otherwise.}
\end{cases}
\end{equation*}
In particular, $\tilde{\rho} \in [-\pi/2 , 0]$.
\end{lemma}

Another technical result that will be used later is the following lemma.

\begin{lemma}[\cite{LO}]\label{lemma2}
Let $t=\tan(\rho/2)$, $\rho\in[0,\pi/2]$, then $\eta_1 = e^{\i \rho}$ and $\eta_2 = e^{\i \tilde\rho}$ are solutions of 
$$ (1-h) + (1+h)s\eta^2_1  + (1+h)(1-s)\eta^2_2 =0 ,$$
if and only if $s\in[h/(1+h), 1/(1+h)]$,
$$ 
\tan\tilde\rho= \frac{-1}{h\tan \rho} \quad \mbox{when}\quad s\in \left(\frac{h}{1+h}, \frac{1}{1+h}\right) ,
$$ 
$\tilde{\rho}=-\pi/2$ for $s=h/(h+1)$ and $\tilde{\rho}=0$ for $s=1/(h+1)$,
and
\begin{equation*}
t = 
\begin{cases}
0,\quad\textnormal{if}\quad s=\frac{h}{1+h}\\
1,\quad\textnormal{if} \quad s=\frac{1}{1+h}\\
\frac{\sqrt{s(1-h^2)} - \sqrt{h(1-s-hs)}}{\sqrt{s-h+hs}},\quad\textnormal{otherwise.}
\end{cases}
\end{equation*}
\end{lemma}

Now, consider a CMC proper-biharmonic immersion $\phi:\mathbb{R}^2\to\mathbb{S}^n$ with mean curvature $h=|H|\in(0,1)$. It is not difficult to check that $\psi=i\circ\phi$ satisfies $\psi(z_1)=\psi(z_2)$ if and only if $\psi(z_1-z_2)=\psi(0)$. Note that $\psi(z)=0$ is equivalent to
$$
\langle \i\sqrt{\lambda_1}\bar\mu_k,z\rangle\equiv 0 \pmod{2\pi},\quad\forall k\in\{1,\ldots,m\}
$$
and
$$
\langle \i\sqrt{\lambda_2}\bar\eta_j,z\rangle\equiv 0 \pmod{2\pi},\quad\forall j\in\{1,\ldots,m'\}.
$$

Next, we define $\Lambda_{\psi}=\{z\in\mathbb{R}^2:\psi(z)=\psi(0)\}$ and note that $\Lambda_{\psi}$ is a discrete lattice and the immersion $\psi$ quotients to an embedding $\psi:\mathbb{R}^2/\Lambda_{\psi}\to\mathbb{R}^{n+1}$ or $\phi:\mathbb{R}^2/\Lambda_{\psi}\to\mathbb{S}^n$.

Now, if $\Lambda$ is a discrete lattice, then $\psi$ quotients to $\mathbb{R}^2/\Lambda$ if and only if $\Lambda\subset\Lambda_{\psi}$, i.e., $\Lambda$ is an abelian subgroup of $\Lambda_{\psi}$. Obviously, if $\rank\Lambda=2$, since $\Lambda\subset\Lambda_{\psi}$, then $\rank\Lambda_{\psi}=2$ and $\mathbb{R}^2/\Lambda$ is a covering space for $\mathbb{R}^2/\Lambda_{\psi}$. Therefore, if $\psi$ quotients to a torus $T^2=\mathbb{R}^2/\Lambda$ then it quotients to a (possible different) torus $T^2=\mathbb{R}^2/\Lambda_{\psi}$ providing an embedding with the same mean curvature.

In the following we will present a new proof to a classification result for CMC proper-biharmonic tori in $\mathbb{S}^5$.

\begin{proposition}[\cite{LO}]\label{prop12}
The CMC proper-biharmonic immersion $\phi_{h,\rho} : \mathbb{R}^2 \to \mathbb{S}^5$, $\rho\in [0,(1/2)\arccos((h-1)/(1+h))]$, quotients to a torus if and only if either
\begin{itemize}

\item[(a)] $\rho=0$ and
$$
h=\frac{1-b}{1+b},
$$
where $b=r^2/t^2$, $r,t\in\mathbb{N}^{\ast}$, with $r<t$ and $(r,t)=1$; or

\item[(b)] $\rho\in (0,(1/2)\arccos((h-1)/(1+h))]$ is a constant depending on $a$ and $b$ and
$$
h = \frac{1- (a-b)^2}{1 + (a-b)^2 + 2(a+b)},
$$
where $a=p^2/q^2$ and $b=r^2/t^2$, with $p,q,r,t \in \mathbb{N}^{\ast}$, such that $0\leq b-a<1$.
\end{itemize}
Moreover, in this case, the corresponding lattice $\Lambda_{\psi_{h,\rho}}$ is given by
\begin{align*}
\Lambda_{\psi_{h,0}}=&\{lrv_2+kv_1:k,l\in\mathbb{Z}\}\\=&\{lt\tilde v_2+kv_1:k,l\in\mathbb{Z}\},
\end{align*}
where $v_1=(-\pi\sqrt{1+b},0)$, $v_2=(0,\pi\sqrt{(1+b)/b})$, and $\tilde v_2=(0,\pi\sqrt{1+b})$; or
\begin{align*}
\Lambda_{\psi_{h,\rho}}=&\left\{ m v_2 + n v_1:m, n \in \mathbb{Z} \mbox{ \upshape{s.t.} } m \frac{q}{p} - n \frac{qr}{pt} \in \mathbb{Z}\right\}\\=&\left\{ m \tilde v_2 + k \tilde v_1 : m, k \in \mathbb{Z} \mbox{ \upshape{s.t.} } m \frac{t}{r} - k \frac{pt}{qr} \in \mathbb{Z}\right\},
\end{align*}
where 
\begin{flalign}\label{eq:v}
v_1&=\left(\frac{\pi\sqrt{(a-b)^2+a+b}}{\sqrt{a}},0\right)\\\nonumber
v_2&=\left(-\frac{\pi\sqrt{b}(1-a+b)}{\sqrt{a((a-b)^2+a+b)}},\frac{\pi\sqrt{1+(a-b)^2+2(a+b)}}{\sqrt{(a-b)^2+a+b}}\right)\\\nonumber
\tilde v_1&=\left(\frac{\pi\sqrt{(a-b)^2+a+b}}{\sqrt{b}},0\right)\\\nonumber
\tilde v_2&=\left(-\frac{\pi\sqrt{a}(1+a-b)}{\sqrt{b((a-b)^2+a+b)}},\frac{\pi\sqrt{1+(a-b)^2+2(a+b)}}{\sqrt{(a-b)^2+a+b}}\right).
\end{flalign}
\end{proposition}

\begin{proof} We will first consider the case when $\rho\in (0,(1/2)\arccos((h-1)/(1+h))]$ and we will explicitly determine $\Lambda_{\psi_{h,\rho}}$ as well as the necessary and sufficient conditions such that $\rank \Lambda_{\psi_{h,\rho}}=2$. 

Consider now $v=(T,V)\in\Lambda_{\psi_{h,\rho}}$ and use the same notations as in Lemma \ref{lemma2}. From $\psi(z)=\psi(0)$ we have
$$
\langle \i\sqrt{\lambda_1},v\rangle=\sqrt{\lambda_1}V=2\pi m,
$$
$$
\langle \i\sqrt{\lambda_2}\bar\eta_1,v\rangle=\sqrt{\lambda_2}(T\sin\rho+V\cos\rho)=2\pi n
$$
and
$$
\langle \i\sqrt{\lambda_2}\bar\eta_2,v\rangle=\sqrt{\lambda_2}(T\sin\tilde\rho+V\cos\tilde\rho)=2\pi k,
$$
where $m,n,k\in\mathbb{Z}$. It follows that $V=(2\pi/\sqrt{\lambda_1})m $ and
$$
T=\frac{2\pi n}{\sqrt{\lambda_2}\sin\rho}-\frac{2\pi m}{\sqrt{\lambda_1}\tan\rho}=\frac{2\pi k}{\sqrt{\lambda_2}\sin\tilde\rho}-\frac{2\pi m}{\sqrt{\lambda_1}\tan\tilde\rho}.
$$
From the two expressions of $T$ one obtains the following condition
\begin{equation}\label{eq:comp1}
nA+mB=k,
\end{equation}
where $A=\sin\tilde\rho/\sin\rho<0$ and $B=-\sqrt{\lambda_2/\lambda_1}A(1+h\tan^2\rho)\cos\rho>0$, that can be rewritten as
\begin{equation}\label{eq:comp2}
k\tilde A+m\tilde B=n,
\end{equation}
with $\tilde A=1/A=\sin\rho/\sin\tilde\rho$ and $\tilde B=-B/A=\sqrt{\lambda_2/\lambda_1}(1+h\tan^2\rho)\cos\rho$.

Therefore, the vector $v$ can be written as 
$$
v=nv_1+mv_2
$$
where
\begin{equation}\label{eq:v1}
v_1=\left(\frac{2\pi}{\sqrt{\lambda_2}\sin\rho},0\right)\quad\mbox{and}\quad v_2=\left(-\frac{2\pi}{\sqrt{\lambda_1}\tan\rho},\frac{2\pi}{\sqrt{\lambda_1}}\right),
\end{equation}
or, equivalently,
$$
v=k\tilde v_1+m\tilde v_2
$$
where
\begin{equation}\label{eq:v2}
\tilde v_1=\left(\frac{2\pi}{\sqrt{\lambda_2}\sin\tilde\rho},0\right)\quad\mbox{and}\quad \tilde v_2=\left(-\frac{2\pi}{\sqrt{\lambda_1}\tan\tilde\rho},\frac{2\pi}{\sqrt{\lambda_1}}\right).
\end{equation}

Since 
$$
\begin{cases}
v_1=A\tilde v_1\\ v_2=\tilde v_2+B\tilde v_1,
\end{cases}
$$
it follows that
\begin{align*}
\Lambda_{\psi_{h,\rho}}=&\left\{mv_2+nv_1:m, n \in \mathbb{Z} \mbox{ \upshape{s.t.} } nA+mB \in \mathbb{Z}\right\}\\=&\left\{m\tilde v_2+k\tilde v_1: m, k \in \mathbb{Z} \mbox{ \upshape{s.t.} } k\tilde A+m\tilde B\in \mathbb{Z}\right\}.
\end{align*}

The next step is to show that $\rank\Lambda_{\psi_{h,\rho}}=2$ if and only if $A,B\in\mathbb{Q}^{\ast}$.

First, assume that $\rank\Lambda_{\psi_{h,\rho}}=2$. This is equivalent to the fact that there exist $w_1,w_2\in\Lambda_{\psi_{h,\rho}}$, given by
$$
w_1=m_1v_2+n_1v_1\quad\mbox{and}\quad w_2=m_2v_2+n_2v_1,
$$
with $n_1A+m_1B=l_1\in\mathbb{Z}$ and $n_2A+m_2B=l_2\in\mathbb{Z}$, such that $\{w_1,w_2\}$ is linearly independent over $\mathbb{R}$.

If $n_1=0$ it follows that $m_1\neq 0$ and $B=l_1/m_1\in\mathbb{Q}^{\ast}$. Also $n_2$ cannot vanish since otherwise $w_2\parallel w_1$. Therefore, we have $A=(l_2-m_2B)/n_2\in\mathbb{Q}^{\ast}$. The other three cases when $n_1m_1n_2m_2=0$ follow in the same way. 

Assume now that $n_1m_1n_2m_2\neq 0$ and we have the following system
$$
\begin{cases}
n_1A+m_1B=l_1\\n_2A+m_2B=l_2,
\end{cases}
$$
in $A$ and $B$. Since $\{w_1,w_2\}$ is linearly independent over $\mathbb{R}$, we see that the discriminant of this system is different from zero, which means that $A$ and $B$ are uniquely determined and $A,B\in\mathbb{Q}^{\ast}$.

Conversely, assume that $A=-p'/q'$ and $B=p''/q''$, where $p',q',p'',q''\in\mathbb{N}^{\ast}$ with $(p',q')=1$ and $(p'',q'')=1$. It is then not hard to see that $\{m(q''v_2)+n(q'v_1):m,n\in\mathbb{Z}\}\subset\Lambda_{\psi_{h,\rho}}$ and $\{q'v_1,q''v_2\}$ is a linearly independent system over $\mathbb{R}$.

To handle $A$ and $B$ it is more convenient to express them in terms of $s$ and $h$. To this end we use Lemma \ref{lemma2}, perform a computation similar to that in the proof of \cite[Proposition~12]{LO}, and obtain
$$
A=-\sqrt{\frac{s(1-s-hs)}{(1-s)(s-h+hs)}}\quad\mbox{and}\quad
B=\sqrt{\frac{h}{(1-s)(s-h+hs)}}.
$$
We have seen that $\rank\Lambda_{\psi_{h,\rho}}=2$ if and only if $A,B\in\mathbb{Q}^{\ast}$, which is equivalent to $1/B,-A/B\in\mathbb{Q}^{\ast}$. These conditions lead to
$$
\begin{cases}
\sqrt{\frac{(1-s)(s-h+hs)}{h}}=\frac{p}{q}\\\sqrt{\frac{s(1-s-hs)}{h}}=\frac{r}{t}
\end{cases},\quad p,q,r,t\in\mathbb{N}^{\ast},\quad (p,q)=1,\quad (r,t)=1.
$$
Denoting $a=p^2/q^2\in\mathbb{Q}^{\ast}_+$ and $b=r^2/t^2\in\mathbb{Q}^{\ast}_+$ the above system can be rewritten as
$$
\begin{cases}
(1+h)s^2-(1+2h)s+(1+a)h=0\\(1+h)s^2-s+bh=0.
\end{cases}
$$
Therefore $A,B\in\mathbb{Q}^{\ast}$ if and only if the last system in $s$ admits (at least) one solution $s\in(h/(1+h),1/2]$, where we have also used the fact that $\rho\in (0,(1/2)\arccos((h-1)/(1+h))]$ if and only if $s\in(h/(1+h),1/2]$, which follows from Lemma \ref{lemma2}. This is equivalent to 
$$
h=\frac{1-(a-b)^2}{1+(a-b)^2+2(a+b)},\quad 0\leq b-a<1,
$$
and the solution is
$$
s=\frac{1+a-b}{2}.
$$

To end the proof we will determine the explicit expression of $\Lambda_{\psi_{h,\rho}}$ in terms of $a$ and $b$. We denote $t=\tan(\rho/2)$, $t\in(0,\sqrt{h+1}-\sqrt{h}]$, and obtain, after a long but straightforward computation,
\begin{align*}
t=&\frac{\sqrt{s(1-h^2)}-\sqrt{h(1-s-hs)}}{\sqrt{s-h+hs}}\\=&\frac{\sqrt{(1+a+b)((a-b)^2+a+b)}-\sqrt{b}(1-a+b)}{\sqrt{a(1+(a-b)^2+2(a+b))}}
\end{align*}
and
$$
\sin\rho=\frac{2t}{1+t^2}=\sqrt{\frac{a(1+(a-b)^2+2(a+b))}{(1+a+b)((a-b)^2+a+b)}},
$$
$$
\cos\rho=\frac{1-t^2}{1+t^2}=\sqrt{\frac{b}{(1+a+b)((a-b)^2+a+b)}}(1-a+b),
$$
$$
\sin\tilde\rho=\frac{t^2-1}{\sqrt{(1-t^2)^2+4h^2t^2}}=-\sqrt{\frac{b(1+(a-b)^2+2(a+b))}{(1+a+b)((a-b)^2+a+b)}},
$$
$$
\cos\tilde\rho=\frac{2ht}{\sqrt{(1-t^2)^2+4h^2t^2}}=\sqrt{\frac{a}{(1+a+b)((a-b)^2+a+b)}}(1+a-b).
$$
Finally, replacing in \eqref{eq:v1} and \eqref{eq:v2}, one obtains the expressions \eqref{eq:v} of $v_1$, $v_2$, $\tilde v_1$, and $\tilde v_2$.

When $\rho=0$, working in the same way as in the previous case, one obtains that the lattice $\Lambda_{\psi_{h,0}}$ is given by
\begin{align*}
\Lambda_{\psi_{h,0}}=&\left\{mv_2+kv_1:m,k \in\mathbb{Z}\quad\mbox{s.t.}\quad m\sqrt{\frac{\lambda_2}{\lambda_1}}\in\mathbb{Z}\right\}\\=&
\left\{n\tilde v_2+kv_1:n,k \in\mathbb{Z}\quad\mbox{s.t.}\quad n\sqrt{\frac{\lambda_1}{\lambda_2}}\in\mathbb{Z}\right\},
\end{align*}
where $v_1=(-2\pi/\sqrt{\lambda_2},0)$, $v_2=(0,2\pi/\sqrt{\lambda_1})$, and $\tilde v_2=(0,2\pi/\sqrt{\lambda_2})$. It is then easy to see that $\rank\Lambda_{\psi_{h,0}}=2$ if and only if $\sqrt{\lambda_2/\lambda_1}\in\mathbb{Q}^{\ast}$, that is equivalent to the fact that the mean curvature $h$ is given by $h=(1-b)/(1+b)$, where $b=r^2/t^2$, $r,t\in\mathbb{N}^{\ast}$, with $r<t$ and $(r,t)=1$. Replacing in the expression of $\Lambda_{\psi_{h,0}}$ we obtain its final form and conclude the proof.
\end{proof}

From Proposition \ref{prop12} one obtains the main result in this section.

\begin{theorem}[\cite{LO}]\label{thm7}
Let $h\in (0,1)$. Then there exists a CMC proper-biharmonic immersion from a torus $T^2$ into $\mathbb{S}^5$, $\phi : T^2\to \mathbb{S}^5$ with mean curvature $h$ if and only if either
\begin{itemize}
\item[(i)] 
$$h = \frac{1-b}{1+b},$$
where $b=r^2/t^2$, $r,t\in\mathbb{N}^{\ast}$, with $r<t$; or
\item[(ii)] 
$$ h = \frac{1- (a-b)^2}{1 + (a-b)^2 + 2(a+b)} ,$$
where $a=p^2/q^2$ and $b=r^2/t^2$, with $p,q,r,t \in \mathbb{N}^{\ast}$, such that $0\leq b-a<1$.
\end{itemize}
\end{theorem}

\section{Biharmonic immersions of tori in spheres}

In this section we consider two types of tori and, using Theorem \ref{thmM}, show how they can be immersed as CMC proper-biharmonic surfaces in odd-dimensional spheres. While our first case, rectangular tori, only works in $\mathbb{S}^5$ and $\mathbb{S}^7$, the second one, square tori, provides examples of CMC proper-biharmonic immersions for any (odd) dimension of the ambient space.

Our next result gives all possible CMC proper-biharmonic immersions of a class of rectangular tori in $\mathbb{S}^n$.

\begin{theorem}\label{thmT1} 
Let $\Lambda=\{(2\pi\tilde k,2\pi\tilde l\theta): \tilde k,\tilde l\in\mathbb{Z}\}$ be a rectangular lattice and consider the torus $T^2=\mathbb{R}^2/\Lambda$, where $\theta\in\mathbb{R}^{\ast}_{+}$. Then $T^2$ admits a proper-biharmonic immersion in $\mathbb{S}^n$ with constant mean curvature $h\in(0,1)$ if and only if 
$$
\theta^2=(q_1^2+q_2^2)/2\quad\mbox{and}\quad n\in\{5,7\},
$$
where $q_1,q_2\in\mathbb{N}$ and $q_1<q_2$. In this case 
$$
h=\frac{q_2^2-q_1^2}{2(q_1^2+q_2^2)},
$$ 
with $q_1\geq 0$ when $n=5$ and $q_1>0$ when $n=7$. Moreover, the CMC proper-biharmonic immersion from $T^2$ to $\mathbb{S}^n$ corresponds to the map $\psi:\mathbb{R}^2\to\mathbb{R}^{n+1}$ given by \eqref{psi} and determined by one of the following sets of data$:$
\begin{itemize}

\item when $n=5$
$$
R_1=1,\quad R'_1=\frac{1}{2}\left(1-\sqrt{\frac{1-2h}{1+2h}}\right),\quad R'_2=\frac{1}{2}\left(1+\sqrt{\frac{1-2h}{1+2h}}\right)
$$ 
$$
\mu_1=\sqrt{\frac{1-2h}{2(1-h)}}+\frac{\i}{\sqrt{2(1-h)}},\quad \eta_1=\sqrt{\frac{1+2h}{2(1+h)}}+\frac{\i}{\sqrt{2(1+h)}},\quad \eta_2=\bar\eta_1,
$$
where $h\in(0,1/2]$;

\item when $n=7$
$$
R_{1,2}=\frac{1}{2}\left(1\pm\frac{\omega}{\sqrt{1-2h}}\right),\quad R'_{1,2}=\frac{1}{2}\left(1\mp\frac{\omega}{\sqrt{1+2h}}\right),
$$
$$
\mu_1=\sqrt{\frac{1-2h}{2(1-h)}}+\frac{\i}{\sqrt{2(1-h)}}, \mu_2=\bar\mu_1,\quad \eta_1=\sqrt{\frac{1+2h}{2(1+h)}}+\frac{\i}{\sqrt{2(1+h)}}, \eta_2=\bar\eta_1,
$$
where $\omega\in(-\sqrt{1-2h},\sqrt{1-2h})$ and $h\in(0,1/2)$.
\end{itemize}
\end{theorem} 

\begin{proof} Let us consider $\theta\in\mathbb{R}^{\ast}_{+}$, the lattice $\Lambda=\{(2\pi\tilde k,2\pi\tilde l\theta): \tilde k,\tilde l\in\mathbb{Z}\}$, and the torus $T^2=\mathbb{R}^2/\Lambda$. Then the spectrum of the Laplacian on $T^2$ is (see \cite{BDKV,BGM}) 
$$
\left\{\lambda=p^2+\frac{q^2}{\theta^2}:p,q\in\mathbb{N}\right\}.
$$

In order for the immersion $\phi:\mathbb{R}^2\to\mathbb{S}^n$ in Theorem \ref{thmM} to quotient to the torus $T^2$ the map $\psi=i\circ\phi:\mathbb{R}^2\to\mathbb{R}^{n+1}$ has to satisfy $\psi(x,y)=\psi(x+2\pi\tilde k,y+2\pi\tilde l\theta)$ for any pairs $(\tilde k,\tilde l)\in\mathbb{Z}^2$ and $(x,y)\in\mathbb{R}^2$. In particular, we have $\psi(0,0)=\psi(2\pi\tilde k,0)=\psi(0,2\pi\tilde l\theta)$, that, using Expression \eqref{psi} of $\psi$, leads to
$$
\cos(2\sqrt{\lambda_1}\pi\tilde kb_k)=\cos(2\sqrt{\lambda_1}\pi\tilde l\theta a_k)=1,\quad \forall\tilde k,\tilde l\in\mathbb{Z}
$$
and
$$
\cos(2\sqrt{\lambda_2}\pi\tilde kd_j)=\cos(2\sqrt{\lambda_2}\pi\tilde l\theta c_j)=1,\quad \forall\tilde k,\tilde l\in\mathbb{Z},
$$
where $\lambda_1$ and $\lambda_2$ are the eigenvalues of $T^2$, $\mu_k=a_k+\i b_k$, and $\eta_j=c_j+\i d_j$, $1\leq k\leq m$, $1\leq j\leq m'$. These conditions are equivalent to
$$
\begin{cases}
\sqrt{\lambda_1}\theta a_k=L_k\in\mathbb{Z}\\\sqrt{\lambda_1} b_k=M_k\in\mathbb{Z}
\end{cases},\quad k\in\{1,\ldots,m\}
$$
and
$$
\begin{cases}
\sqrt{\lambda_2}\theta c_j=L'_j\in\mathbb{Z}\\\sqrt{\lambda_2} d_j=M'_j\in\mathbb{Z}
\end{cases},\quad j\in\{1,\ldots,m'\}.
$$
Therefore, we have
$$
\mu_k=\frac{L_k}{\theta\sqrt{\lambda_1}}+\i\frac{M_k}{\sqrt{\lambda_1}}\quad\textnormal{and}\quad
\eta_j=\frac{L'_j}{\theta\sqrt{\lambda_2}}+\i\frac{M'_j}{\sqrt{\lambda_2}}.
$$

It can be easily verified that, with these data, the map $\psi:\mathbb{R}^2\to\mathbb{R}^n$ satisfies $\psi(x,y)=\psi(x+2\pi\tilde k,y+2\pi\tilde l\theta)$ for any $(\tilde k,\tilde l)\in\mathbb{Z}^2$ and $(x,y)\in\mathbb{R}^2$ and therefore it quotients to the torus $T^2$.

Now, since $\lambda_1=2(1-h)$, $\lambda_2=2(1+h)$ and $|\mu_k|=|\eta_j|=1$, one obtains
\begin{equation}\label{eq:2}
\begin{cases}
L_k^2+M_k^2\theta^2=2\theta^2(1-h)\\
(L'_j)^2+(M'_j)^2\theta^2=2\theta^2(1+h)
\end{cases},\quad 1\leq k\leq m, \quad 1\leq j\leq m'.
\end{equation}

Next, from Condition (f) in Theorem \ref{thmM}, we get 
\begin{equation}\label{eq:1}
\begin{cases}
\displaystyle\sum_{k=1}^m R_k\left(L_k^2-M_k^2\theta^2\right)+\sum_{j=1}^{m'} R'_j\left((L'_j)^2-(M'_j)^2\theta^2\right)=0\\
\displaystyle\sum_{k=1}^m R_kL_kM_k+\sum_{j=1}^{m'} R'_jL'_jM'_j=0.
\end{cases}
\end{equation}

It is easy to see, from Equations \eqref{eq:2}, that $M_k^2<2$ and $(M_j')^2<4$, for $1\leq k\leq m$, $1\leq j\leq m'$, and we shall study all possible cases:
\begin{itemize}

\item $M_k^2=(M'_j)^2=0$, for all indices $k$ and $j$. From equations \eqref{eq:2} and the second equation of \eqref{eq:1} it follows $\theta=0$, which is a contradiction.

\item $M_k^2=1$ and $(M'_j)^2=0$ or $M_k^2=0$ and $(M'_j)^2=1$, for all indices $k$ and $j$. From Equations \eqref{eq:2} and the first equation of \eqref{eq:1} one obtains again that $\theta$ must vanish.

\item Assume that there exist $s,s'\in\mathbb{N}^{\ast}$, with $s\leq m$, $s'\leq m'$, and $s+s'<m+m'$, such that
$$
M_1^2=\cdots=M_s^2=0,\quad M_{s+1}^2=\cdots=M_m^2=1
$$
and
$$
(M'_1)^2=\cdots=(M'_{s'})^2=0,\quad (M'_{s'+1})^2=\cdots=(M'_{m'})^2=1.
$$
Then, again using Equations \eqref{eq:2} and since $\theta\neq 0$, the first equation of \eqref{eq:1} becomes 
$$
\sum_{k=1}^{s}R_k-h\sum_{k=1}^m R_k+\sum_{j=1}^{s'}R'_j+h\sum_{j=1}^{m'}R'_j=0,
$$
that is $\sum_{k=1}^{s}R_k+\sum_{j=1}^{s'}R'_j=0$, which is a contradiction.

\item $M_k^2=(M'_j)^2=1$, for all indices $k$ and $j$. In this case Equations \eqref{eq:1} are satisfied.
\end{itemize}  
We have just shown that $M_k,M'_j\in\{-1,1\}$, $L^2_k=\theta^2(1-2h)$, and $(L'_j)^2=\theta^2(1+2h)$, $1\leq k\leq m$, $1\leq j\leq m'$. Moreover, the mean curvature $h$ can only take values in the interval $(0,1/2]$. Also, since we must have $m$ distinct $\mu_k$'s and $m'$ distinct $\eta_j$'s, it is easy to see that $(m,m')\in\{(1,2),(2,1),(2,2)\}$, which means that $n=2(m+m')-1\in\{5,7\}$. 

Next, we will find admissible forms for $\theta$ and $h$. First, we know that the eigenvalues of $T^2$ take the form
$$
\lambda_1=2(1-h)=p_1^2+\frac{q_1^2}{\theta^2}\quad\textnormal{and}\quad\lambda_2=2(1+h)=p_2^2+\frac{q_2^2}{\theta^2},
$$
where $p_1,p_2,q_1,q_2\in\mathbb{N}$. Since $\lambda_1+\lambda_2=4$, it follows that $q_1^2+q_2^2>0$, $4-p_1^2-p_2^2>0$, and 
$$
\theta^2=\frac{q_1^2+q_2^2}{4-p_1^2-p_2^2}.
$$
Also, since $\lambda_2-\lambda_1=2h$, we have
$$
h=\frac{2(q_2^2-q_1^2)+p_2^2q_1^2-p_1^2q_2^2}{2(q_1^2+q_2^2)}.
$$

The condition $4-p_1^2-p_2^2>0$ leaves only four possible cases for the pair $(p_1^2,p_2^2)$ that will be analyzed in the following.
\begin{itemize}
\item If $p_1^2=p_2^2=0$ then $0<q_1^2<q_2^2$ and
$$
\theta^2=\frac{q_1^2+q_2^2}{4},\quad h=\frac{q_2^2-q_1^2}{q_1^2+q_2^2}.
$$
Therefore, we have
$$
L_k^2=\frac{3q_1^2-q_2^2}{4},\quad (L'_j)^2=\frac{3q_2^2-q_1^2}{4}.
$$
Now, since $q_1\neq 0$, we can write $q_1=a\alpha$, with $a\in\mathbb{N}^{\ast}$, such that $L_k=a\alpha_k$, with the greatest common divisor $(\alpha,\alpha_k)$ equal to $1$. It follows that $q_2^2=a^2(3\alpha^2-4\alpha_k^2)$ and, therefore, that $3\alpha^2-4\alpha_k^2$ must be a perfect square. But, since a perfect square can only be of the form $9t^2$ or $3t+1$, with $t\in\mathbb{N}$, and $(\alpha,\alpha_k)=1$, by a straightforward analysis of all possible cases, one easily shows that this is not possible.

\item The cases $(p_1^2,p_2^2)=(1,0)$ or $(0,1)$ can be dismissed in the same way as above.

\item When $p_1^2=p_2^2=1$, one obtains $0\leq q_1^2<q_2^2$ and
$$
\theta^2=\frac{q_1^2+q_2^2}{2},\quad h=\frac{q_2^2-q_1^2}{2(q_1^2+q_2^2)}.
$$
Then, we have $L_k^2=q_1^2$, and $(L'_j)^2=q_2^2$.
\end{itemize}

When $n=5$ we have $(m,m')\in\{(1,2),(2,1)\}$. If $(m,m')=(1,2)$ then one obtains 
$$
R_1=1,\quad \mu_1=\sqrt{\frac{1-2h}{2(1-h)}}+\frac{\i}{\sqrt{2(1-h)}},\quad \eta_1=\sqrt{\frac{1+2h}{2(1+h)}}+\frac{\i}{\sqrt{2(1+h)}},\quad \eta_2=\bar\eta_1,
$$
and, from the second equation of \eqref{eq:1}, 
$$
R_1'-R_2'=-\sqrt{\frac{1-2h}{1+2h}},
$$
which, together with $R_1'+R_2'=1$, leads to 
$$
R_{1,2}'=\frac{1}{2}\left(1\mp\sqrt{\frac{1-2h}{1+2h}}\right).
$$
By the symmetry of solutions the other possible cases determine no new CMC proper-biharmonic immersions in $\mathbb{S}^5$.

If $(m,m')=(2,1)$ then we have
$$
R'_1=1,\quad \mu_1=\sqrt{\frac{1-2h}{2(1-h)}}+\frac{\i}{\sqrt{2(1-h)}},\quad \mu_2=\bar\mu_1,\quad \eta_1=\sqrt{\frac{1+2h}{2(1+h)}}+\frac{\i}{\sqrt{2(1+h)}}
$$
and obtain 
$$
R_1=\frac{1}{2}\left(1-\sqrt{\frac{1+2h}{1-2h}}\right)<0,
$$
that is a contradiction, which was to be expected as we have already seen in the previous section.

When $n=7$ we take 
$$
\mu_1=\sqrt{\frac{1-2h}{2(1-h)}}+\frac{\i}{\sqrt{2(1-h)}}, \mu_2=\bar\mu_1,\quad \eta_1=\sqrt{\frac{1+2h}{2(1+h)}}+\frac{\i}{\sqrt{2(1+h)}},\eta_2=\bar\eta_1,
$$
and, since $R_1+R_2=1$ and $R_1'+R_2'=1$, the second equation of \eqref{eq:1} becomes
$$
(R_1-R_2)\sqrt{1-2h}+(R'_1-R'_2)\sqrt{1+2h}=0,
$$
with solutions
$$
R_{1,2}=\frac{1}{2}\left(1\pm\frac{\omega}{\sqrt{1-2h}}\right),\quad R'_{1,2}=\frac{1}{2}\left(1\mp\frac{\omega}{\sqrt{1+2h}}\right),
$$
where $\omega\in(-\sqrt{1-2h},\sqrt{1-2h})$ and $h\in(0,1/2)$.
\end{proof}

\begin{remark} A torus $T^2$ of the type considered in Theorem \ref{thmT1} admits a CMC proper-biharmonic immersion with $h=1/2$ only in $\mathbb{S}^5$. In this case, $\theta=q_2^2/2$ and the immersion is given by
$$
R_1=1,\quad R'_1=R'_2=\frac{1}{2},\quad \mu_1=\i,\quad \eta_1=\frac{\sqrt{2}+\i}{\sqrt{3}},\quad \eta_2=\bar\eta_1.
$$
\end{remark}

\begin{remark}\label{r3.3} When $n=5$ the expression of $h$ in Theorem \ref{thmT1} can be obtained from the general expression of $h$ in Proposition \ref{prop12}, case (b), by taking $a=((q_1-q_2)/(2q_2))^2$ and $b=((q_1+q_2)/(2q_2))^2$, with $0\leq q_1<q_2$. It is easy to see that $0\leq b-a<1$. For these values of $a$ and $b$ we have that 
$$
v_1=\left(\frac{\pi\sqrt{2}\sqrt{3q_1^2+q_2^2}}{q_2-q_1},0\right)\quad\mbox{and}\quad v_2=\left(-\frac{\pi\sqrt{2}(q_1+q_2)^2}{(q_2-q_1)\sqrt{3q_1^2+q_2^2}},\frac{2\pi\sqrt{q_1^2+q_2^2}}{\sqrt{3q_1^2+q_2^2}}\right),
$$
and it is not hard to see that $v\in\Lambda_{\psi_{h,\rho}}$ if and only if $v=mv_2+nv_1$, with $m,n\in\mathbb{Z}$ and 
$$
\frac{2mq_2-n(q_1+q_2)}{q_2-q_1}\in\mathbb{Z}.
$$
On the other hand, the lattice $\Lambda$ in Theorem \ref{thmT1} is generated by
$$
f_1=(2\pi,0)\quad\mbox{and}\quad f_2=(0,2\pi\theta)=\left(0,\pi\sqrt{2}\sqrt{q_1^2+q_2^2}\right).
$$
Consider a rotation in $\mathbb{R}^2$ with matrix
$$
\left(\begin{array}{cc}
\frac{q_1\sqrt{2}}{\sqrt{3q_1^2+q_2^2}}&\frac{\sqrt{q_1^2+q_2^2}}{\sqrt{3q_1^2+q_2^2}}\\ \\-\frac{\sqrt{q_1^2+q_2^2}}{\sqrt{3q_1^2+q_2^2}}&\frac{q_1\sqrt{2}}{\sqrt{3q_1^2+q_2^2}}\end{array}\right)
$$
and still denote the resulting lattice by $\Lambda$. The vectors $f_1$ and $f_2$ become
$$
\tilde f_1=\left(\frac{2\pi\sqrt{2}q_1}{\sqrt{3q_1^2+q_2^2}},-\frac{2\pi\sqrt{q_1^2+q_2^2}}{\sqrt{3q_1^2+q_2^2}}\right)\quad\mbox{and}\quad\tilde f_2=\left(\frac{\pi\sqrt{2}(q_1^2+q_2^2)}{\sqrt{3q_1^2+q_2^2}},\frac{2\pi q_1\sqrt{q_1^2+q_2^2}}{\sqrt{3q_1^2+q_2^2}}\right),
$$
respectively, and it can easily be verified that
$$
\tilde f_1=-v_1-v_2\in\Lambda_{\psi_{h,\rho}}\quad\mbox{and}\quad \tilde f_2=q_2v_1+q_1v_2\in\Lambda_{\psi_{h,\rho}},
$$
which means that $\Lambda\subseteq\Lambda_{\psi_{h,\rho}}$. It is also easy to see that if $q_2=q_1+1$, then $\Lambda=\Lambda_{\psi_{h,\rho}}$and therefore we do not merely have an immersion but an embedding.
\end{remark}

\begin{remark} For $0<q_1^2<q_2^2$ the same torus can be immersed in $\mathbb{S}^5$ and in $\mathbb{S}^7$ as CMC proper-biharmonic surfaces, with the same constant mean curvature.
\end{remark}

\begin{remark} Let $r\in\mathbb{N}^{\ast}$ and $\tilde q_1=rq_1$, $\tilde q_2=rq_2$ and consider $\tilde\theta^2=(\tilde q_1^2+\tilde q_2^2)/2$. Then we have a family of non-isometric rectangular tori immersed in $\mathbb{S}^5$ (or in $\mathbb{S}^7$) as CMC proper-biharmonic surfaces with the same mean curvature. However, if we consider two pairs $(q_1,q_2)$ and $(\tilde q_1,\tilde q_2)$ such that $h=\tilde h$, then the pairs of rational numbers $(a,b)$ and $(\tilde a,\tilde b)$ given by Remark \ref{r3.3} coincide. Therefore the corresponding lattices $\Lambda_{\psi_{h,\rho}}$ coincide.
\end{remark}

\begin{remark} The same rectangular torus can be immersed in $\mathbb{S}^5$ (or in $\mathbb{S}^7$) as a CMC proper-biharmonic surface in different ways with different mean curvatures.
\end{remark}

\begin{remark} From the proof of Theorem \ref{thmT1} it is easy to see that a rectangular torus with both sides of length less than $1/\sqrt{2}$ cannot be immersed in a sphere $\mathbb{S}^n$ as a CMC proper-biharmonic surface.
\end{remark}

The following result provides necessary and sufficient conditions for square tori to admit CMC proper-biharmonic immersions in odd-dimensional spheres. In order to state the theorem we first denote by $r_2(p)$ the number of representations of $p\in\mathbb{N}$ as a sum of two squares of integers.

\begin{theorem}\label{thmT2}
Let $\Lambda=\{(2\pi\tilde ka,2\pi\tilde la): \tilde k,\tilde l\in\mathbb{Z}\}$ be a square lattice and consider the torus $T^2=\mathbb{R}^2/\Lambda$, where $a\in\mathbb{R}^{\ast}_{+}$. Then we have
\begin{itemize}
\item[(a)] $T^2$ admits a proper-biharmonic immersion in $\mathbb{S}^n$, $n\equiv 3 \pmod 4$, with constant mean curvature $h\in(0,1)$ if and only if 
$$
4a^2=p_1^2+q_1^2+p_2^2+q_2^2,\quad h=\frac{p_2^2+q_2^2-p_1^2-q_1^2}{p_1^2+q_1^2+p_2^2+q_2^2},
$$
and
$$
7\leq n\leq r_2(p_1^2+q_1^2)+r_2(p_2^2+q_2^2)-1,
$$
where $p_1,q_1,p_2,q_2\in\mathbb{N}$ such that $0<p_1^2+q_1^2<p_2^2+q_2^2$. 

\item[(b)] If $4a^2=p^2+q^2$, where $p,q\in\mathbb{N}$ such that $0<p<q$, then $T^2$ admits a CMC proper-biharmonic immersion in $\mathbb{S}^n$ with $h=(q^2-p^2)/(p^2+q^2)$ for any odd $n$, $5\leq n\leq r_2(p^2)+r_2(q^2)-1$.
\end{itemize}
\end{theorem}

\begin{proof} The immersion $\phi:\mathbb{R}^2\to\mathbb{S}^n$ in Theorem \ref{thmM} quotients to the square torus $T^2$ if and only if the map $\psi=i\circ\phi:\mathbb{R}^2\to\mathbb{R}^{n+1}$ satisfies $\psi(x,y)=\psi(x+2\pi\tilde ka,y+2\pi\tilde la)$ for any $(\tilde k,\tilde l)\in\mathbb{Z}^2$ and $(x,y)\in\mathbb{R}^2$. Therefore, we have $\psi(0,0)=\psi(2\pi\tilde k,0)=\psi(0,2\pi\tilde l\theta)$, which, also using the expression \eqref{psi} of $\psi$, shows that
$$
\cos(2\sqrt{\lambda_1}\pi\tilde kab_k)=\cos(2\sqrt{\lambda_1}\pi\tilde la a_k)=1,\quad \forall\tilde k,\tilde l\in\mathbb{Z}
$$
and
$$
\cos(2\sqrt{\lambda_2}\pi\tilde kad_j)=\cos(2\sqrt{\lambda_2}\pi\tilde la c_j)=1,\quad \forall\tilde k,\tilde l\in\mathbb{Z},
$$
where $\lambda_1$ and $\lambda_2$ are the eigenvalues of $T^2$, $\mu_k=a_k+\i b_k$, and $\eta_j=c_j+\i d_j$, $1\leq k\leq m$, $1\leq j\leq m'$.

Thus, one obtains
$$
\mu_k=\frac{L_k+\i M_k}{a\sqrt{\lambda_1}}\quad\textnormal{and}\quad
\eta_j=\frac{L'_j+\i M'_j}{a\sqrt{\lambda_2}},
$$
where $L_k,M_k,L'_j,M'_j\in\mathbb{Z}$, $1\leq k\leq m$, $1\leq j\leq m'$. 

It is easy to verify that these conditions also imply the general periodicity property $\psi(x,y)=\psi(x+2\pi\tilde ka,y+2\pi\tilde la)$, $(\tilde k,\tilde l)\in\mathbb{Z}^2$ and $(x,y)\in\mathbb{R}^2$, and, therefore, the immersion $\phi:\mathbb{R}^2\to\mathbb{S}^n$ quotients to $T^2$.

Next, since $\lambda_1=2(1-h)$, $\lambda_2=2(1+h)$, and $|\mu_k|=|\eta_j|=1$, we have
\begin{equation}\label{eq:22}
\begin{cases}
L_k^2+M_k^2=2a^2(1-h)\\
(L'_j)^2+(M'_j)^2=2a^2(1+h)
\end{cases},\quad 1\leq k\leq m, \quad 1\leq j\leq m'.
\end{equation}

Also, condition (f) in Theorem \ref{thmM} gives
\begin{equation}\label{eq:12}
\begin{cases}
\displaystyle\sum_{k=1}^m R_k\left(L_k^2-M_k^2\right)+\sum_{j=1}^{m'} R'_j\left((L'_j)^2-(M'_j)^2\right)=0\\
\displaystyle\sum_{k=1}^m R_kL_kM_k+\sum_{j=1}^{m'} R'_jL'_jM'_j=0.
\end{cases}
\end{equation}
If we denote 
$$
U=\left(\sqrt{R_1}L_1,\ldots,\sqrt{R_m}L_m,\sqrt{R'_1}L'_1,\ldots\sqrt{R'_{m'}}L'_{m'}\right)
$$ 
and 
$$
V=\left(\sqrt{R_1}M_1,\ldots,\sqrt{R_m}M_m,\sqrt{R'_1}M'_1,\ldots\sqrt{R'_{m'}}M'_{m'}\right),
$$
the above system can be written as
$$
\begin{cases}
||U||=||V||\\\langle U,V\rangle=0.
\end{cases}
$$

We will determine in the following the admissible forms for $a$ and $h$.

Since the spectrum of the Laplacian on $T^2$ is (see \cite{BDKV,BGM})
$$
\left\{\lambda=\frac{p^2+q^2}{a^2}:p,q\in\mathbb{N}\right\},
$$
it follows that there exist four nonnegative integers $p_1$, $q_1$, $p_2$, and $q_2$, with $0<p_1^2+q_1^2<p_2^2+q_2^2$, such that $\lambda_1=(p_1^2+q_1^2)/a^2$ and $\lambda_2=(p_2^2+q_2^2)/a^2$. Now, it is not hard to see that condition $\lambda_1+\lambda_2=4$ implies that
$$
4a^2=p_1^2+q_1^2+p_2^2+q_2^2\quad\mbox{and}\quad h=\frac{p_2^2+q_2^2-p_1^2-q_1^2}{p_1^2+q_1^2+p_2^2+q_2^2}.
$$

To determine the maximum possible dimension of the ambient space we first note that, from Equations \eqref{eq:22}, one obtains
\begin{equation}\label{eq:3.5}
\begin{cases}
L_k^2+M_k^2=p_1^2+q_1^2\\
(L'_j)^2+(M'_j)^2=p_2^2+q_2^2
\end{cases},\quad 1\leq k\leq m, \quad 1\leq j\leq m',
\end{equation}
that shows that $0<2m\leq r_2(p_1^2+q_1^2)$ and $0<2m'\leq r_2(p_2^2+q_2^2)$. Then, from \cite[Theorem~4.12]{SS} and Conditions (g) and (h) in Theorem \ref{thmM}, we get that
$$
n\leq r_2(p_1^2+q_1^2)+r_2(p_2^2+q_2^2)-1=4N-1,
$$
where $N=(r_2(p_1^2+q_1^2)+r_2(p_2^2+q_2^2))/4\in\mathbb{N}$, $N\geq 2$.

Let us assume that $n=4N-1$. In order to obtain a particular solution of the system formed from Equations \eqref{eq:12} and \eqref{eq:3.5}, we first note that $m$ and $m'$ are even numbers and take $R_k=1/m$, $1\leq k\leq m$, and $R'_j=1/m'$, $1\leq j\leq m'$. We consider $L_{2\alpha-1}$, $M_{2\alpha-1}$, $L'_{2\beta-1}$, $M'_{2\beta-1}$ nonnegative integers such that $L^2_{2\alpha-1}+M^2_{2\alpha-1}=p_1^2+q_1^2$, $(L'_{2\beta-1})^2+(M'_{2\beta-1})^2=p_2^2+q_2^2$, and
$$
\begin{cases}
L_{2\alpha}=M_{2\alpha-1},\quad M_{2\alpha}=-L_{2\alpha-1}\\L'_{2\beta}=M'_{2\beta-1},\quad M'_{2\beta}=-L'_{2\beta-1}
\end{cases},\quad 1\leq\alpha\leq \frac{m}{2},\quad 1\leq\beta\leq \frac{m'}{2}.
$$
If $\sqrt{p_1^2+q_1^2}\in\mathbb{N}^{\ast}$, then we consider the pair $(L_{2\alpha-1},M_{2\alpha-1})=(0,\sqrt{p_1^2+q_1^2})$ but not $(\sqrt{p_1^2+q_1^2},0)$.

We now have $m=r_2(p_1^2+q_1^2)/2$ distinct unit complex numbers 
$$
\mu^2_k=\left(\frac{L_k+\i M_k}{a\sqrt{2(1-h)}}\right)^2
$$
and $m'=r_2(p_2^2+q_2^2)/2$ distinct unit complex numbers 
$$
\eta^2_j=\left(\frac{L'_j+\i M'_j}{a\sqrt{2(1+h)}}\right)^2.
$$ 
Then the data
$$
\left(R_k=\frac{1}{m},R'_j=\frac{1}{m'},\mu_k=\frac{L_k+\i M_k}{a\sqrt{2(1-h)}},\eta_j=\frac{L'_j+\i M'_j}{a\sqrt{2(1+h)}}\right),
$$
with $1\leq k\leq m$, $1\leq j\leq m'$, $2(m+m')-1=n=4N-1\geq 7$, determine a proper-biharmonic immersion of $T^2$ in $\mathbb{S}^n$ with constant mean curvature $h$. 

Next, it can be easily verified that, for example, the data
$$
\left(R_k=\frac{1}{m-2},R'_j=\frac{1}{m'},\mu_1,\dots,\mu_{m-2},\eta_1,\ldots,\eta_{m'}\right)
$$
also give a CMC proper-biharmonic immersion of $T^2$ in $\mathbb{S}^{4N-5}$, with the same mean curvature $h$. Working this way, we can construct CMC proper-biharmonic immersions of $T^2$ in any sphere $\mathbb{S}^n$, $n\in\{7,\ldots,4N-5,4N-1\}$, all with the same mean curvature $h$.   

In the following, we shall assume that $4a^2=p^2+q^2$, with $p,q\in\mathbb{N}$, $0<p<q$. Let us also consider $m$ odd, such that $m<r_2(p^2)/2$, and $m'$ even, such that $m'=r_2(q^2)$. Then we can take $p_1=p$, $q_1=0$, $p_2=0$, and $q_2=q$ and obtain $h=(q^2-p^2)/(p^2+q^2)$. One can check that another solution of \eqref{eq:12} and \eqref{eq:3.5} is given by
$$
L_1=a\sqrt{2(1-h)}=p,\quad M_1=0,
$$
the positive integers $L_k$, $M_k$, with $2\leq k\leq m$, chosen such that $L_k^2+M_k^2=2a^2(1-h)=p^2$ and 
$$
L_{2\alpha+1}=-M_{2\alpha},\quad M_{2\alpha+1}=L_{2\alpha},\quad 1\leq\alpha\leq \frac{m-1}{2},
$$
and
$$
\begin{cases}
\displaystyle R_1=\frac{1}{2},\quad R_2=\cdots=R_m=\frac{1}{2(m-1)},\quad\mbox{if}\quad m>1\\
\displaystyle R_1=1,\quad\mbox{if}\quad m=1,
\end{cases} 
$$
and also
$$
L'_1=a\sqrt{2(1+h)}=q,\quad M'_1=0,\quad L'_2=M'_1=0,\quad M'_2=-L'_1=-q
$$
and positive integers $L'_j$, $M'_j$, with $3\leq j\leq m'$, chosen such that $(L'_j)^2+(M'_j)^2=2a^2(1+h)=q^2$ and 
$$
L'_{2\beta}=M'_{2\beta-1},\quad M'_{2\beta}=-L'_{2\beta-1},\quad 2\leq\beta\leq \frac{m'}{2},
$$
and
$$
\begin{cases}
\displaystyle R'_{1,2}=\frac{1}{4}\mp\frac{1-h}{4(1+h)},\quad R'_3=\cdots=R'_{m'}=\frac{1}{2(m'-2)},\quad\mbox{if}\quad m'>2\\
\displaystyle R'_{1,2}=\frac{1}{2}\mp\frac{1-h}{2(1+h)},\quad\mbox{if}\quad m'=2.
\end{cases}
$$

For this solution we have $0<2m\leq r_2(p^2)-2$ and $0<2m'=r_2(q^2)$ and, therefore, again using \cite[Theorem~4.12]{SS}, one obtains $5\leq n\leq r_2(p^2)+r_2(q^2)-3=r_2(p_1^2+q_1^2)+r_2(p_2^2+q_2^2)-3=4N-3$. 

Hence, we have a CMC proper-biharmonic immersion of $T^2$ in $\mathbb{S}^n$, $n=4N-3\geq 5$, with mean curvature $h$. In the same way as for the other particular solution, using this immersion, one can construct proper-biharmonic immersions of $T^2$ in any $\mathbb{S}^n$, with $n\in\{5,\ldots,4N-7,4N-3\}$, with the same constant mean curvature.

Therefore, these explicit examples show that we have CMC proper-biharmonic immersions of $T^2$ in $\mathbb{S}^n$ for any odd integer $n$ between or equal to $5$ and $4N-1$.
\end{proof}

\begin{remark} From the proof of Theorem \ref{thmT2}, it is not hard to see that $a\geq\sqrt{3}/2$. Moreover, if $a=\sqrt{3}/2$, one obtains that the corresponding square torus can be immersed only in $\mathbb{S}^7$, in a unique way given by
$$
\left(R_1=R_2=R'_1=R'_2=\frac{1}{2},\mu_1=i,\mu_2=1,\eta_1=\frac{1+\i}{\sqrt{2}},\eta_2=\bar\eta_2\right).
$$
\end{remark}

\begin{remark} While any positive integer can be written as a sum of four squares (not necessarily satisfying the condition in Theorem \ref{thmT2} though), a positive integer can be written as a sum of two squares if and only if each of its prime factors of the form $4p-1$ occurs with an even power in its prime factorization (see, for example, \cite{SS}).
\end{remark}

As positive integers $p$ and $q$ can be chosen such that $r_2(p^2)+r_2(q^2)$ is arbitrarily large, from Theorem \ref{thmT2}, we have the following result.

\begin{theorem} For any sphere $\mathbb{S}^n$, with $n$ odd, there exists a square torus that can be immersed in $\mathbb{S}^n$ as a CMC proper-biharmonic surface.
\end{theorem}


\begin{thebibliography}{99}

\bibitem{BDKV} C.~Baikoussis, F.~Defever, T.~Koufogiorgos, L.~Verstraelen, \textit{Finite type immersions of flat tori into Euclidean spaces}, Proc. Edinburgh Math. Soc. (2) 38 (1995), 413--420. 

\bibitem{BMO} A. Balmu\c s, S. Montaldo, C. Oniciuc, {\it Classification results for biharmonic submanifolds
in spheres}, Israel J. Math. 168 (2008), 201--220.

\bibitem{BGM} M. Berger, P. Gauduchon, E. Mazet, \textit{Le spectre d'une vari\'et\'e riemannienne}, Lecture Notes in Mathematics 194, Springer-Verlag, Berlin-New York, 1971.

\bibitem{C} B.-Y. Chen, \textit{Some open problems and conjectures on submanifolds of finite type}, Soochow J. Math. 17 (1991), 169--188.

\bibitem{ES} J.~Eells, J. H.~Sampson, \textit{Harmonic mappings of Riemannian manifolds}, Amer. J. Math. 86 (1964), 109--160.

\bibitem{J} G. Y.~Jiang, \textit{$2$-harmonic maps and their first and second variational formulas}, Chinese Ann. Math. Ser. A7(4) (1986), 389--402.

\bibitem{LO} E.~Loubeau, C.~Oniciuc, \textit{Constant mean curvature proper-biharmonic surfaces of constant Gaussian curvature in spheres}, J. Math. Soc. Japan 68 (2016), 997--1024.

\bibitem{M} Y.~Miyata, \textit{$2$-Type surfaces of constant curvature in $\mathbb{S}^n$}, Tokyo J. Math. 11 (1988), 157--204.

\bibitem{O} C.~Oniciuc, \textit{Tangency and Harmonicity Properties}, PhD Thesis, Geometry Balkan Press 2003, \texttt{http://www.mathem.pub.ro/dgds/mono/dgdsmono.htm}

\bibitem{SS} J.~D.~Sally, P.~J.~Sally, Jr., \textit{Roots to research. A vertical development of mathematical problems}, American Mathematical Society, Providence, RI, 2007.


\end{thebibliography}
\end{document}